\documentclass[11pt,a4paper]{article}

\usepackage[T1]{fontenc}
\usepackage[utf8]{inputenc}
\usepackage[cm]{aeguill}
\usepackage[english]{babel}

\usepackage[all]{xy}
\usepackage{amsmath}
\usepackage{amssymb}
\usepackage{amsthm}
\usepackage{enumerate}

\usepackage{graphicx}
\usepackage{psfrag}

\usepackage{geometry}

\newtheorem{thm}{Theorem}[section]
\newtheorem{defn}[thm]{Definition}
\newtheorem{prop}[thm]{Proposition}
\newtheorem{lemma}[thm]{Lemma}
\newtheorem{cor}[thm]{Corollary}
\newtheorem{rem}[thm]{Remark}

\newcommand{\Aa}{{\mathcal A}}
\newcommand{\Bb}{{\mathcal B}}

\newcommand{\Hh}{{\mathcal H}}

\newcommand{\Ll}{{\mathcal L}}

\newcommand{\Tt}{{\mathcal T}}

\newcommand{\NM}{{\mathbb N}}

\newcommand{\RM}{{\mathbb R}}

\newcommand{\ZM}{{\mathbb Z}}

\newcommand{\TR}{{\rm Tr\,}}                       

\renewcommand{\phi}{\varphi}                  

\renewcommand{\tilde}{\widetilde}

\newcommand{\wee}{\, \tilde\wedge \,}       

\newcommand{\dtau}{d_\tau}        
\newcommand{\tdtau}{\tilde d_\tau}        

\newcommand{\ind}{O}    

\newcommand{\dsup}{d_{\text{\rm sup}}}   
\newcommand{\dinf}{d_{\text{\rm inf}}}   
\newcommand{\tdsup}{\tilde{d}_{\text{\rm sup}}}   
\newcommand{\tdinf}{\tilde{d}_{\text{\rm inf}}}   
\newcommand{\op}{\text{\rm op}}     
\newcommand{\btau}{b_\tau}           
\newcommand{\tbtau}{\tilde b_\tau}   

\newcommand{\prs}{p_{\text{\rm rs}}} 
\newcommand{\ppr}{p_{\text{\rm pr}}} 
\newcommand{\brs}{\beta_{\text{\rm rs}}} 
\newcommand{\bpr}{\beta_{\text{\rm pr}}} 

\title{A characterization of subshifts with bounded powers}
\author{J. Kellendonk$^{1}$, D. Lenz$^{2}$, J. Savinien$^{3}$\\
{\small $^{1}$ Universit\'e Lyon I, Lyon, France.} \\
{\small $^{2}$ Friedrich Schiller Universität, Jena, Germany.} \\
{\small $^{3}$ Université de Lorraine, Metz, France.}
}

\begin{document}

\maketitle

\begin{abstract}
We consider minimal, aperiodic  symbolic subshifts and show
how to characterize the  combinatorial property of bounded powers  by means of a metric property.
For this purpose we construct a family of graphs which all approximate the subshift space, and define a metric on each graph which extends to a metric on the subshift space.
The characterization of bounded powers is then  given by the Lipschitz equivalence of a suitably defined infimum
metric with the corresponding  supremum metric.
We also introduce zeta-functions and relate their abscissa of convergence to various exponents of complexity of the subshift.
\end{abstract}

\section{Introduction}
\label{sect-intro}

In symbolic dynamics one studies subshifts of the so-called full
shift over a finite alphabet $\Aa$; the latter is the $\ZM$-action
given by the left shift $\sigma$ on the set of infinite sequences with
values in $\Aa$ and a subshift is the restriction of this dynamical
system to a closed shift invariant subspace $\Xi$. Among the fields of
interest are the combinatorial properties of such subshifts. The most prominent combinatorial 
properties occurring in the literature are  recurrence and its stronger
variant linear recurrence,
repulsiveness which is equivalent to bounded powers (also referred
to as power freeness), richness, and various forms of complexity.
Such combinatorial properties often correspond to properties of the
dynamical system and hence of the $C^*$-algebras $C(\Xi)$ and
$C(\Xi)\rtimes_\sigma\ZM$. So it is a natural idea to consider non
commutative Riemannian geometries \cite{Co94}[Chap.~VI], that is, spectral
triples, on these
algebras and see how these can be used to characterize combinatorial
properties of the subshift.

While spectral triples for crossed product algebras of the above type
seem hard to set up - we are only aware of the recent attempt
\cite{BMR} which only gives a partial result,
and a version for the related crossed product with $\RM$
\cite{Whi} which seems very implicit - there has been quite
some activity in constructing spectral triples for commutative
$C^*$-algebras $C(X)$ whose space $X$ does not carry an obvious
differential Riemannian structure.
A series of works has been devoted to metric
spaces \cite{Ri99,Ri04,CI07} or more specifically to fractals
\cite{GI03,GI05,CIL08}
and Cantor sets \cite{Co94}. In particular, for ultrametric Cantor
sets the work of Pearson \& Bellissard \cite{Pea08,PB09}
can be regarded as a mile
stone. They introduced and emphasized the importance of choice
functions.

In recent work \cite{KS10}, two of the authors proposed a modification
of Pearson \& Bellissard's triple obtaining in particular
a characterization of the  combinatorial property of bounded powers  for subshifts with a unique
right-special word per length. A subshift has
bounded powers if its sequences do not contain arbitrarily high powers
of words, {\it i.e.} there is an integer $p$ such that $n$-fold
repetitions $w^n=w\cdots w$ of a word $w$ cannot occur for $n>p$.
Note that linearly  recurrent subshifts, which are commonly regarded as highly
ordered \cite{LP03, Dur00, Dur03}, share this property.  A subshift has a unique right-special
word per length if, for each $n$, there exists a unique word of
length $n$ which can be extended to the right in more
than one way to a word of length $n+1$.
The purpose of the present work is to generalize this characterization
of bounded powers to the whole class of minimal and aperiodic
subshifts.

The essential ingredient in the construction of \cite{KS10} is a
family of graphs which approximate the subshift: its vertices are
dense and its edges encode adjacencies.
Each graph gives rise to spectral triple and its associated Connes distance, and taking extrema over the family yields two metrics on the subshift space.
The result is then
that the subshift has bounded
powers if and only if the two metrics are Lipschitz equivalent.
The generalization to all subshifts given in the present work  is based on the use of
{\it a priori} different approximation graphs. These are obtained by trading
right-special words, which played a decisive role for the old graphs, against
what we call here {\em privileged words}.

Privileged words are iterated complete first returns to letters of
the alphabet. They have met a lot of  interest recently. For the class of rich subshifts
the privileged words are exactly the palindromes (see
Section~\ref{pf11.ssect-spwords} for further details).

As is often the case that, once one is lead to consider certain objects by
an abstract theory (here non commutative Riemannian geometry) and
these objects turn out useful in the context of another field (here
subshifts) one finds out that they can also be defined {\it ad hoc},
i.e.\ without any knowledge of the abstract theory.
This is the case here and so we present our construction {\it ad hoc} and add a final section
in which we explain the spectral triples underlying it.

\vspace{.5cm}

The paper is organized as follows:
We recall basic definitions about subshifts in
Section~\ref{pf11.sect-subshifts}.
We explain bounded powers and repulsiveness, we introduce {\em privileged
words} and explain their relation to palindromes
(Proposition~\ref{pf11.prop-icr}), and define subshifts of {\em almost
finite ranks}.

Section~\ref{pf11.sect-treegraph} is devoted to the construction of
the approximation graphs. For that we first
recall the definition of the {\em tree of words} $\Tt$ of a
right-infinite subshift $\Xi$.
We introduce two types of {\em horizontal edges}: one type for
right-special words and another for privileged words
(Definition~\ref{def-H} and~\ref{def-tH}). The above mentioned main
result of this work will make use only of privileged horizontal edges
but for comparison with \cite{KS10} we consider
right-special horizontal edges as well. Similar to \cite{PB09}
and as in \cite{KS10}, {\em choice functions}  (Definition~\ref{pf11.def-choices}) will play
a role to define the approximation graphs for the subshift space
and a {\em weight function} will be used to give a length to the horizontal edges.

In Section~\ref{pf11.sect-metric}, we define {\it ad hoc} a metric on $\Xi$ by
\[
\tdtau (\xi, \eta): = \sup_{f\in C(\Xi)}
\Bigl\{
|f(\xi)-f(\eta)| \, : \, |f(s(e)) - f(r(e))| \le l(e), \, \forall e
\in
\tilde E_\tau
\Bigr\}
\]
where $s(e)$ and $r(e)$ denote the source and range vertex of the edge
$e$, $l(e)$ its length,
and $\tilde E_\tau$ the realization of the horizontal edges of the
approximation graph defined
by the choice function $\tau$.
We provide an explicit formula for $\tilde d_\tau$
in Lemma~\ref{pf11.lem-specdist}. We define the extremal metrics
$\tdinf$ and $\tdsup$ and derive explicit criteria for their Lipschitz
equivalence. We also compare the above metrics with the metrics which were
obtained in \cite{KS10} (Prop.~\ref{prop-old-new}).

In Section~\ref{pf11.sect-powerfree} we state and prove our main result:
\bigskip

\noindent {\bf Theorem~\ref{pf11.thm-characterization}} {\em
Let $\Xi$ be a minimal and aperiodic $\ZM$-subshift over a finite alphabet.
Then $\Xi$ has bounded powers if and only if $\tdsup$ and $\tdinf$ are
Lipschitz equivalent. }
\smallskip

In Section~\ref{pf11.sect-Zeta} we introduce two families of {\em zeta-functions}.
These are defined by Dirichlet series and their summability is related to various exponents of complexity of the subshift.

In the last Section~\ref{pf11.sect-spectrip} we briefly explain the non commutative geometrical constructions underlying this work.
We provide the spectral triple associated to an approximation graph, show that the associated Connes distance is $\tdtau$, and relate the zeta-function of the spectral triple to the zeta-functions defined in Section~\ref{pf11.sect-Zeta}.

\paragraph{Acknowledgments}
This work was supported by the ANR grant {\em SubTile} no. NT09 564112.
The authors would like to thank Luca Zamboni for useful discussions;
in particular he explained them the notion of rich words and showed them Proposition~\ref{pf11.prop-icr}.

\section{Subshifts}
\label{pf11.sect-subshifts}

A subshift is a subspace \(\Xi \subset \Aa^\ZM\) of sequences over a finite alphabet $\Aa$, that is closed (for the product topology) and invariant under the left-shift map $\sigma$.
A (finite) word occurring in some infinite word $\xi \in \Xi$ is called a {\em factor}.
The set $\Ll$ of all factors of all $\xi \in \Xi$ is called the {\em language} of the subshift.
We consider subshifts that are {\em aperiodic}: \(\forall \xi \in \Xi, \, \sigma^n(\xi) = \xi \Rightarrow n=0\), and for which the dynamical system given by the action of $\ZM$ by the shift is {\em minimal} (every orbit is dense).

The length of a word $u$ is written $|u|$. 
Given $u,v \in \Ll$, we write \(v \preceq u\) to mean that $v$ is a prefix of $u$, and $v\prec u$ if $v$ is a proper prefix ({\it i.e.} $|v|<|u|$).
Similarly we write \( u \succeq v\) or \(u \succ v\) if $v$ is a suffix or proper suffix of $u$.

\subsection{Bounded powers}
\label{pf11.ssect-bdpowers}

A subshift $\Xi$ has {\em bounded powers} if there exists an integer $p$
such that any word can occur at most $p$ times consecutively: \(
\forall u \in \Ll,\; u^{p+1} \notin \Ll\). This is sometimes also
called power free.

The following characterization of bounded powers will be useful.
Define the {\em index of repulsiveness} of a subshift $\Xi$ with
language $\Ll$ as
\begin{equation}
\label{pf11.eq-indexrepuls}
\ell := \inf\Big\{
\frac{|W| - |w|}{|w|} \, : \, w, W \in \Ll, \;
w \; \text{\rm is a proper prefix and suffix of } W
\Big\}\,.
\end{equation}
A subshift is called {\em repulsive} if $\ell>0$.
\begin{lemma}
\label{pf11.rem-powers}
A subshift with has bounded powers if and only if it is repulsive.
\end{lemma}
\begin{proof}
If $\Xi$ has arbitrarily large powers, for all integer $p$ there exists a word $u \in \Ll$ such that $u^p \in \Ll$.
Take $w=u^{p-1}$ and $W=u^p$ in equation~\eqref{pf11.eq-indexrepuls}, to get $\ell \le 1/(p-1)$.
Since this must hold for any $p$, we conclude that $\ell = 0$.
Conversely, if $\ell=0$, then for any $\epsilon>0$ arbitrarily small, there exists words $w,W \in \Ll$ as in equation~\eqref{pf11.eq-indexrepuls} such that the ratio \((|W|-|w|)/|w|\) is less than $\epsilon$.
This implies that the two occurrences of $w$ in $W$ overlap, and in turns that one can write $w=u^{p-1}v$ and $W=u^pv$ for some $u,v\in \Ll$ with $0<|v|\le |u|$, and with $p$ greater than or equal to the integer part of $1/\epsilon$.
Hence $\Xi$ has arbitrarily large powers.
\end{proof}

One defines a right- of left-infinite subshift similarly as a subset \(\Xi \subset \Aa^\NM\) of right- or left-infinite sequences.
Given a subshift $\Xi$ one denotes by $\Xi^\pm$ the right- and left-infinite subshifts derived from $\Xi$ (by dropping the left or right parts of infinite words in $\Xi$).

\begin{lemma}
\label{pf11.lem-repulsive}
Let $\Xi$ be a minimal and aperiodic subshift.
The following assertions are equivalent:
\begin{enumerate}[(i)]
\item $\Xi$ has bounded powers;

\item $\Xi^+$ has bounded powers;

\item $\Xi^-$ has bounded powers.
\end{enumerate}
\end{lemma}
\begin{proof}
Since the three subshifts have the same language, the indices of repulsiveness of $\Xi^\pm$ are equal to that of $\Xi$: $\ell^\pm=\ell$.
\end{proof}

\subsection{Privileged words}
\label{pf11.ssect-spwords}

We consider a minimal and aperiodic {\em right-infinite} subshift $\Xi$ with language $\Ll$ over a finite alphabet.
As a consequence of minimality, given a word $u\in \Ll$, there exists finitely many non-empty words $u' \in \Ll$, called {\em complete first return} words to $u$, such that
\begin{enumerate}[(i)]
\item $u$ is a prefix and a suffix of $u'$,

\item $u$ occurs exactly twice in $u'$.
\end{enumerate}
If $u$ is the empty word, its complete first returns are by definition the letters of the alphabet.
An $n$-th iterated complete first return of $u$ is a word $u^{(n)}$ for which there exists words \(u^{(j)}, j=0, \cdots n-1\), such that $u^{(0)}=u$ and $u^{(j+1)}$ is a complete first return to $u^{(j)}$, for \(j=0, \cdots n-1\).
An $n$-th iterated complete first return word $u$ of the empty word will be called an $n$-th order {\em privileged word}, and we will denote by $\ind(u)=n$ its order.
So for instance the unique $0$-th order privileged word is the empty word, and the $1$-th order privileged words are the letters of the alphabet.
\medskip

We say that a subshift has {\em finite privileged rank} if there is a finite number $N$ such that
any privileged word $u$ has only finitely many complete first return words $u'$. Using Bratteli Vershik diagram techniques \cite{HPS}
to describe the subshift, based on a Kakutani-Rohlin towers whose bases are cylinder sets of privileged words (see \cite{Du10}),
one easily sees that this
implies
that the rationalized \v{C}ech-cohomology of the subshift space is finite generated. We will need a generalization: We say that a subshift has {\em almost finite privileged rank} if there are constants $a,b>0$ such that the number of complete first return words of a privileged word $u$
is bounded by $ a \log(|u|)^b$.

\medskip
We now show the relation between privileged words and palindromes.
An infinite word $\xi$ is called {\em rich} \cite{GJWZ09} if any factor $u$ of $\xi$ contains exactly $|u|+1$ palindromes.
The notion of privileged words is a ``maximal generalization'' of palindromes: indeed one can easily see that any factor $u$ of any infinite word contains exactly $|u|+1$ privileged words.
A characteristic property of rich words (\cite{BLGZ09} Proposition 1) is that any complete first return to a palindrome is a palindrome.
\begin{prop}
\label{pf11.prop-icr}
Let $\xi$ be an infinite word over a finite alphabet, and $u$ a factor of $\xi$.
\begin{enumerate}[(i)]
\item If $u$ is a palindrome then it is a privileged word.

\item If $\xi$ is rich, then $u$ is a palindrome if and only if $u$ is a privileged word.
\end{enumerate}
\end{prop}
\begin{proof}
We prove this by induction on $|u|$. The statements are trivial if $|u|=0,1$.

(i) Choose a palindrome $u$, with $|u|>1$, and assume that the statement holds for any word of length less than $|u|$.
Let $v$ be the largest proper palindromic prefix of $u$.
Since $u$ is a palindrome, $v$ is also a suffix of $u$.
Now by maximality of $|v|$, $v$ can only occur twice in $u$.
Hence $u$ is a complete first return of $v$, and therefore a palindrome.

(ii) Choose a privileged word $u$, with $|u|>1$, and assume that the statement holds for any word of length less than $|u|$.
Let $v$ be the privileged word to which $u$ is the complete first return word (note that $v$ is unique).
As $|v|<|u|$, $v$ is a palindrome, and therefore $u$ is a palindrome (as a complete first return to a palindrome).
\end{proof}

\medskip

A word $u\in \Ll$ is called {\em right-special} if it has more than one one-letter right extension: \(\exists a,b \in \Aa,\; a\neq b, \; ua, ub\in \Ll\).
If for all $n\in\NM$ the subshift has a unique right-special word of length $n$, one says that the subshift has {\em a unique right-special word per length}.

Given a word $u$ we denote by $S(u)$ the set of all right-special words $r$, for which there exists a complete first return $u'$ to $u$ such that \( u \preceq r \prec u'\).

\begin{lemma}
The following assertions are equivalent:
\begin{enumerate}[(i)]
\item Given a privileged word $u$ and any complete first return $u'$ to $u$, there exists a unique right-special word $r$ such that \( u \preceq r \prec u'\); 

\item Given a right-special word $r$ and the smallest proper right-special extension $r'$ of $r$, there exists a unique privileged word $u$ such that \( r \preceq u \prec r'\); 

\item Given a privileged word $u$, $S(u)$ contains exactly one (right-special) element.

\end{enumerate}
\end{lemma}
\begin{proof}
Equivalence of the first two conditions follows easily from aperiodicity, and the fact that if $u$ is privileged and $u'$ a complete first return to $u$ then there exists no privileged word $v$ such that \(u\prec v \prec u'\).
The third condition clearly implies the first.
Suppose the first and consider $u'_1,u'_2$, two different complete first returns to $u$.
Then the unique right-special word between $u$ and $u'_1$ coincides with that between $u$ and $u'_2$.
It follows that $S(u)$ contains only one element.
\end{proof}

We call a subshift satisfying the above equivalent conditions {\em right-special balanced}.
The following lemma shows that subshifts studied in \cite{KS10} are right-special balanced.
\begin{lemma}
If a subshift has a unique right-special word per length then it is right-special balanced.
\end{lemma}
\begin{proof}
Let $u'$ be a complete first return to $u$ and $r_1,r_2$ two  right-special words satisfying $u\preceq r_1\prec r_2 \prec u'$.
By uniqueness of right-special factors of length $|r_1|$, $r_1$ must be a suffix of $r_2$.
Hence, if $r_2 \neq r_1$, then $r_2$ is a non-trivial complete first return to $r_1$ and thus contains a non-trivial complete first return to $u$, which is a contradiction.
\end{proof}

\section{Trees and graphs}
\label{pf11.sect-treegraph}

We consider a minimal and aperiodic {\em right-infinite} subshift $\Xi$ over a finite alphabet $\Aa$, with language $\Ll$.

\subsection{The tree of words}
\label{pf11.ssect-tree}

As in \cite{KS10} we consider the tree of words \(\Tt=(\Tt^{(0)},\Tt^{(1)})\): the vertices are the words in $\Ll$ (the root being the empty word), and there is an edge linking a word to each of its one-letter right extension.
The set of infinite rooted paths $\Pi_\infty$ on $\Tt$ can be seen as a subset of $\Aa^\NM$ and shall be equipped with the relative topology of the product topology on $\Aa^\NM$.
It is well known that  $\Pi_\infty$ is homeomorphic to $\Xi$ and hence we identify the two.
In fact, the cylinder sets $[v]$, of all infinite rooted paths through $v\in \Tt^{(0)}$, form a basis of clopen (closed and open) sets for the topology.
Let us denote by $H^{(0)}$ the set of right-special words and by $\tilde H^{(0)}$ the set of privileged words.
It is clear that the above base of the  topology is given by $\{[v]:v\in H^{(0)}\}$.
\begin{lemma}
\label{pf11.lem-infinitepaths}
The cylinder sets $[v]$ for $v\in \tilde H^{(0)}$  also form a basis of clopen sets for the topology.
\end{lemma}
\begin{proof}
Fix a word $u \in \Ll$, and let $v_1$ be its first (left) letter.
Consider the complete first return $v_2$ of $v_1$ which is a prefix of $u$.
Let $v_3$ be the complete first return word of $v_2$ which is a prefix of $u$, and so on.
We define this way a finite sequence $v_1, v_2, \cdots v_{p}$ of elements in $\tilde H^{(0)}$, such that \(v_1 \prec v_2 \prec \cdots v_{p-1} \preceq u \prec v_p\).
Identifying the cylinders $[v], v\in \Tt^{(0)}$, with cylinders of $\Xi$, we have the inclusions \([v_{p-1}] \subset [u] \subset [v_{p}]\) which proves the homeomorphism.
\end{proof}

Given two distinct infinite words $\xi, \eta \in \Xi$, we denote by
\[
\begin{array}{cl}
\xi \wedge \eta\in H^{(0)}\,, & \text{\rm the longest common prefix to $\xi$ and $\eta$, and by} \\
\xi \wee \eta\in \tilde H^{(0)}\,, & \text{\rm the longest common privileged prefix to $\xi$ and $\eta$}\,.
\end{array}
\]
Notice that $\xi\wee\eta$ is always a prefix of $\xi \wedge \eta$.

\subsection{Horizontal edges}
\label{pf11.ssect-horizontal}

\begin{defn}
\label{pf11.def-av}
For $v\in\Tt^{(0)}$ define:
\begin{enumerate}[(i)]
\item $a(v)=$  number of one-letter right extensions of $v$ minus one;

\item $\tilde a(v)=$ number of complete first returns to $v$ minus one if $v$ is privileged, and $0$ if $v$ is not privileged.

\end{enumerate}
\end{defn}

Note that $0\leq a(v)\leq |\Aa|-1$, and $a(v)\geq 1$ whenever $v$ is right-special.
By aperiodicity, for all $n$ there is at least one $v$ of length $n$ such that $a(v)\geq 1$.
Aperiodicity also implies that $\tilde a(v)\geq 1$ for all privileged words.
The following relation between the two definitions will be useful later on.
\begin{lemma}
\label{lem-av}
If $u$ is privileged then
\[
\tilde a(u) = \sum_{r\in S(u)} a(r)\,.
\]
In particular $\tilde a(u)$ bounds the number of right-special words in $S(u)$.
\end{lemma}
\begin{proof} The proof is rather straightforward.
Figure~\ref{pf11.fig-av} illustrates the idea of the proof: the white square stands for a privileged word $u$, the white circles for its complete first returns, and the black circles for the right-special words in $S(u)$.
\begin{figure}[!h]
\begin{center}
\includegraphics[width=9cm]{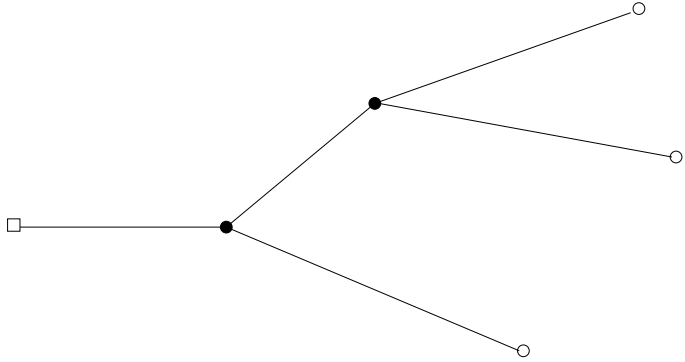}
\end{center}
\caption{\small{Illustration for the sum in Lemma~\ref{lem-av}.}}
\label{pf11.fig-av}
\end{figure}
\end{proof}

\bigskip

The following set has also been used in \cite{KS10}.
\begin{defn}
\label{def-H}
Let $H^{(1)}$ be the set of pairs $(u,v)$ given by distinct one-letter right extensions of the same word (necessarily right-special).
We view these as new edges in the graph $\Tt$ calling them {\em right-special horizontal edges}.
We denote by $u\wedge v$ the corresponding right-special word (the longest common prefix of $u$ and $v$).
\end{defn}
Note that $H^{(1)}$ contains $a(r)(a(r)+1)$ edges with longest common
prefix $r$.
The data $(\Tt^{(0)},\Tt^{(1)}, H^{(1)})$ together with a choice
function and a weight function determine a metric on $\Xi$, as we
recall below, and gave rise to the characterization of power
boundedness in \cite{KS10} in the case of when $\Xi$ has a unique
right-special word per length.

The main new idea in this article is to use another set of horizontal edges.
\begin{defn}
\label{def-tH}
Let $\tilde H^{(1)}$ be the set of pairs $(u,v)$ given by distinct complete first return words of the same privileged word.
We view these as new edges in the graph $\Tt$ calling them {\em privileged horizontal edges}.
We denote by $u\wee v$ the corresponding privileged word (the longest common privileged prefix of $u$ and $v$).
\end{defn}
As for infinite words, $u \wee v$ is always a prefix of $u \wedge v$.

The new general characterization of power freeness will be obtained from the data
$(\Tt^{(0)},\Tt^{(1)}, \tilde H^{(1)})$.

\begin{rem}
\label{pf11.rem-Gambaudo}{\em
The horizontal data $\tilde H^{(0)}$ and $\tilde H^{(1)}$ can be made into a new graph, by adding vertical edges linking a privileged word to any of its complete first returns.
This ``graph of privileged words'' can then be interpreted as a symbolic analogous of a general construction for tilings and Delone sets of $\RM^d$ introduced by Gambaudo {\it et al.} in \cite{BBG06}.
}
\end{rem}

There are natural maps:
\[
\phi^{(0)}:\tilde H^{(0)}\to H^{(0)} \,, \qquad \phi^{(1)}:\tilde H^{(1)}\to H^{(1)}\,,
\]
defined as follows.
Given a privileged word $u$, $\phi^{(0)}(u)$ is the shortest right-special word containing $u$ as a prefix (which, by minimality, always exists).
Given $(u_1,u_2)\in\tilde H^{(1)}$, $u_1\wedge u_2$ is a right-special word and there is a unique one-letter extension $v_i$ of $u_1\wedge u_2$ which is a prefix of $u_i$, $i=1,2$.
We define $\phi^{(1)}((u_1,u_2)) = (v_1,v_2)$.
\begin{lemma}
\label{lem-phi}
The map $\phi^{(0)}$ is always injective.
It is surjective if and only if the subshift is right-special balanced.
For any $(u_1,u_2)\in \tilde H^{(1)}$ we have:
\[
\phi^{(0)}(u_1\wee u_2) = u_1\wedge u_2\,.
\]
Furthermore, if the subshift is right-special balanced then \(a(\phi^{(0)}(u)) = \tilde a(u)\).

The map $\phi^{(1)}$ always surjective.
It is injective if and only if the subshift is right-special balanced.
\end{lemma}
\begin{proof} The statements concerning $\phi^{(0)}$ are obvious.

That right-special balanced implies injectivity is a simple counting argument following from the fact that $a(\phi^{(0)}(u))=\tilde a(u)$ in that case.
As for the converse, if $S(u)$ contains two distinct $r_1,r_2$ then it must contain two distinct $r_1,r_2$ with $r_1\prec r_2$.
It follows that there are distinct complete first returns $u'_1,u'_2,u'_3$ of $u$ such that $r_1$ is the longest common prefix of them all but $r_2$ is the longest common prefix of $r_2$ and $r_3$ only.
It follows that $\phi^{(1)}((u_1,u_2))=\phi^{(1)}((u_1,u_3))$.
\end{proof}

\medskip

An important technical point for this paper is the following lemma: it says that the set of privileged words keeps track of the combinatorics of powers in the subshift.
\begin{lemma}
\label{pf11.lem-powers}
Consider a word $u \in \Ll$.
If there exists an integer $p\ge 2$ such that $u^p\in\Ll$, then there are $p$ non-empty privileged words \(v_1, v_2, \cdots v_p\), and a prefix $\tilde{u}$ of $u$, satisfying
\begin{enumerate}[(i)]

\item $u^p$ is a proper prefix of $v_p$,

\item \(v_j = u^j \tilde{u}\), for \(j=1, 2, \cdots p-1\),

\item $v_{j+1}$ is a complete first return to $v_j$, for \(j=1, 2, \cdots p-1\).

\end{enumerate}
\end{lemma}
\begin{proof}
Let $v_p$ be the shortest privileged proper extension of $u^p$, and let $v_{p-1}$ be the (unique) privileged word whose complete first return is $v_p$.
By minimality of $|v_p|$, $v_{p-1}$ is a prefix of $u^p$, so we have \(v_{p-1} \preceq u^p \prec v_p\).
Hence there is a prefix $\tilde{u}$ of $u$ such that $v_{p-1}=u^k\tilde{u}$ for some $k\le p-1$.
If $k<p-1$, then the first complete first return to $v_{p-1}$, {\it i.e.} $v_p$, would be shorter than $u^p$, a contradiction. Thus we have $v_{p-1} = u^{p-1} \tilde{u}$.

Consider now the (unique) privileged word $v_{p-2}$ whose complete first return is $v_{p-1}$.
The same reasoning, namely that its first complete first return $v_{p-1}$ must be longer than $u^{p-1}$, shows that $v_{p-2} = u^{p-2} \tilde{u}'$, for some prefix $\tilde{u}'$ of $u$.
But $v_{p-2}$ is also a suffix of $v_{p-1}$, and hence $\tilde{u}'=\tilde{u}$.
And we complete the proof with a finite induction.
\end{proof}

\subsection{Approximation graphs}
\label{pf11.ssect-approxgraph}

We consider a minimal and aperiodic right-infinite subshift $\Xi$ with language $\Ll$ over a finite alphabet, and its tree of words $\Tt=(\Tt^{(0)},\Tt^{(1)})$ and the horizontal structures $H$ and $\tilde H$ as defined in the previous Sections~\ref{pf11.ssect-tree} and~\ref{pf11.ssect-horizontal}.

\begin{defn}
\label{pf11.def-choices}
A {\em choice function} is a map \( \tau : \Tt^{(0)} \rightarrow \Pi_\infty\) which satisfies
\begin{enumerate}[(i)]
\item $\tau(v)$ goes through $v$,
\item If $\tau(v)$ goes through $w$, with $|w|>|v|$, then $\tau(w) = \tau(v)$.
\end{enumerate}
\end{defn}
Given a choice function $\tau$ we define the {\em approximation graphs} \(\Gamma_\tau = (V,E)\) and \(\tilde\Gamma_\tau = (\tilde V,\tilde E)\)
by
\[
V = \tau(H^{(0)})\,, \qquad  \qquad
E =  \bigl\{ \bigl(\tau(u) ,\tau(v) \bigr) \, : \, (u,v) \in H^{(1)} \bigr\} \,,
\]
and
\[
\tilde V = \tau(\tilde H^{(0)})\,, \qquad \qquad
\tilde E =  \bigl\{ \bigl(\tau(u) ,\tau(v) \bigr) \, : \, (u,v) \in \tilde H^{(1)} \bigr\} \,.
\]
Given an edge $e=(\xi,\eta)$ in $E$ or $\tilde E$, we write $s(e)=\xi$ and $r(e)=\eta$ for its source and range vertices, and $e^\op = (\eta,\xi)$ for its opposite edge.

Notice that $\Gamma_\tau$ and $\tilde \Gamma_\tau$ are both connected graphs.

The graph $\Gamma_\tau$ was introduced in \cite{KS10}.
For the class of subshifts studied in \cite{KS10}, the two graphs are the same.
\begin{prop}
\label{pf11.prop-apgraph}
If the subshift is right-special balanced then $\Gamma_\tau=\tilde\Gamma_\tau$.
\end{prop}
\begin{proof}
For all subshifts, $\Gamma_\tau$ and $\tilde\Gamma_\tau$ have the same vertices.
We need to show that for all $(u_1,u_2)\in\tilde H^{(1)}$ there are $(v_1,v_2)\in H^{(1)}$ such that $\tau(u_i) = \tau(v_i), i=1,2,$ and vice versa.
By Lemma~\ref{lem-phi}, $\phi^{(1)}$ induces a bijection between the two types of horizontal edges.
By the second property of choice functions we have \((\tau\times\tau)\circ\phi^{(1)} = \tau\times\tau\).
\end{proof}
\bigskip

We now introduce a weight function which will be used to define a
metric on the graphs.
\begin{defn}
\label{def-weight}
A {\em weight function} is a strictly decreasing function $\delta:\ZM \to\RM^+$
which tends to $0$ at infinity and for which there exist constants
$\overline{c},\underline{c}>0$ such that
\begin{enumerate}[(i)]

\item $\delta(ab)\leq \overline{c}\delta(a)\delta(b)$,

\item $\delta(2a)\geq \underline{c}\delta(a)$.

\end{enumerate}
\end{defn}
Our characterization will not depend on the choice of weight function.
So the reader may simply choose one so that $\delta(n) = \frac{1}{n+1}$ for $n\in\NM$ to get the usual word metric below in Remark~\ref{pf11.rem-approxgraph} (ii).

Given a weight function $\delta$ we associate the following length to the horizontal edges:
\[
l((u,v)) = \left\{ \begin{array}{ll}
 \delta(|u\wedge v|)\, & (u,v)\in H^{(1)} \,,\\
 \delta(|u \wee v|)\, & (u,v)\in \tilde H^{(1)}\,.
\end{array} \right.
\]
We have the following elementary inequalities, on $H^{(0)}$ and $\tilde H^{(1)}$ respectively:
\[
\delta \circ \phi^{(0)}\leq \delta\,, \qquad \text{\rm and } \qquad l \circ \phi^{(1)} \leq l \,.
\]

The length function allows us to define a graph metric on $\Gamma_\tau$ and $\tilde\Gamma_\tau$:
\[
d_g (\xi,\eta) = \inf \sum_{j=1}^n l(e_j)\,, \  \xi, \eta \in V\,, \qquad
\tilde d_g (\xi,\eta) = \inf \sum_{j=1}^n l(e_j)\,, \  \xi, \eta \in \tilde V\,,
\]
the infimum running over all (finite) sequences $(e_j)_{1\leq j\leq n}$ of edges in $E$ or $\tilde E$ such that \(s(e_1)=\xi, \cdots r(e_j)=s(e_{j+1}), \cdots r(e_n)=\eta\).

\begin{rem}
\label{pf11.rem-approxgraph}
\begin{enumerate}[(i)]
\item We call $\Gamma_\tau$ and $\tilde\Gamma_\tau$ approximation graphs because $V$ and $\tilde V$ are dense in $\Xi$, and $E$ and $\tilde E$ encode neighboring infinite words.

Indeed, since $\tau$ picks an infinite word for each cylinder $[v]$, $v$ in $H^{(0)}$ or $\tilde H^{(0)}$, {\it i.e.} for each basis clopen set for the topology of $\Xi$ by Lemma~\ref{pf11.lem-infinitepaths}, we see that $V$ and $\tilde V$ are dense in $\Xi$. 
Now given $e = (\xi, \eta)$ in $E$ or $\tilde E$, both $\xi$ and $\eta$ belong to the cylinder $[\xi \wedge \eta]$ or $[\xi \wee \eta]$, and can thus be considered ``neighbors'' (see the next item).

\item The function $\delta$ allows us to define metrics $d$ and $\tilde d$ on $\Xi$ as follows:
\begin{equation}
\label{pf11.eq-metric}
d(\xi, \eta) = \left\{ \begin{array}{ll} \delta(|\xi\wedge \eta|) & \text{\rm if } \xi \neq\eta \,, \\
0 & \text{\rm if } \xi = \eta \,. \end{array} \right. \qquad
\tilde d(\xi, \eta) = \left\{ \begin{array}{ll} \delta(|\xi\wee \eta|) & \text{\rm if } \xi \neq\eta \,, \\
0 & \text{\rm if } \xi = \eta \,. \end{array} \right.
\end{equation}
Notice that $d$ and $\tilde d$ actually define ultrametrics on $\Xi$.
Now $x\wee y$ is always a prefix of $x\wedge y$, so we have
\[
d(\xi, \eta) \le \tilde d(\xi, \eta)\,,\quad  \forall \xi,\eta \in \Xi\,,
\]
and
\[
d(\xi, \eta) \le  d_g(\xi, \eta)\,, \quad \text{\rm and } \quad
\tilde d(\xi, \eta) \le \tilde d_g(\xi, \eta)\,,\qquad \xi,\eta \in  V=\tilde V\,.
\]

\end{enumerate}
\end{rem}

\section{Metrics}
\label{pf11.sect-metric}

\subsection{Metrics associated to the approximation graphs}
\label{pf11.ssect-metrics}

The construction given in \cite{KS10} of a metric on the subshift
space followed the recipes of spectral triples. Indeed, the length
function on the edges the graph $\Gamma_\tau$ gives rise
to a spectral triple so that the famous Connes-formula yields a
metric (the spectral distance) which extends to $\Xi$.
The situation is analogous with $\tilde \Gamma_\tau$ as we now show.
\begin{defn}
\label{pf11.def-metrics}
We define two metrics on $\Xi$:
the metric $d_\tau$ given by:
\begin{equation}\label{def-m1}
\dtau (\xi, \eta) = \sup_{ f\in C(\Xi)} \bigl\{ |f(\xi) - f(\eta)| \, :\,
\forall e\in E, \; |f(s(e))-f(r(e))| \le l(e)
\bigr\}\,,
\end{equation}
and the metric $\tilde d_\tau$ given by:
\begin{equation}\label{def-m2}
\tdtau (\xi, \eta) = \sup_{ f\in C(\Xi)} \bigl\{ |f(\xi) - f(\eta)| \, :\,
\forall e\in \tilde E, \; |f(s(e))-f(r(e))| \le l(e)
\bigr\}.
\end{equation}
\end{defn}

Given an infinite word $\xi \in \Xi$, we denote by $\xi_n$ its $n$-th right-special prefix, and by $\tilde \xi_n$ its $n$-th order privileged prefix.
We define
\[
\btau(\xi_n) = \left\{ \begin{array}{cl}  1 & \text{\rm if } \tau(\xi_n) \wedge \xi = \xi_n \,,\\
0 & \text{\rm else} \,,
\end{array} \right. \qquad \text{\rm and } \qquad
\tbtau(\tilde\xi_n) = \left\{ \begin{array}{cl}  1 & \text{\rm if } \tau(\tilde\xi_n) \wee \xi = \tilde\xi_n \,,\\
0 & \text{\rm else} \,,
\end{array} \right.
\]
which we use to provide explicit formulas for $\dtau$ and $\tilde\dtau$.
\begin{lemma}
\label{pf11.lem-specdist}
The metrics $\dtau$ and $\tdtau$ are extensions of the graph metrics $d_g$ and $\tilde d_g$, on $\Gamma_\tau$ and $\tilde\Gamma_\tau$, respectively.
For $\xi,\eta\in \mbox{\rm Im} (\tau)$ they are given by
\begin{equation}
\label{eq-dtau}
d_\tau(\xi, \eta) = \delta(|\xi\wedge\eta|) +
\sum_{n>|\xi\wedge\eta|} \btau(\xi_n) \delta(|\xi_n|)
+ \sum_{n>|\xi\wedge\eta|} \btau(\eta_n) \delta(|\eta_n|)\,,
\end{equation}
\begin{equation}
\label{eq-tdtau}
\tdtau(\xi, \eta) = \delta(|\xi\wee\eta|) +
\sum_{n> \ind(\xi\wee\eta)} \tbtau(\tilde \xi_n) \delta(|\tilde \xi_n|)
+ \sum_{n>\ind(\xi\wee\eta)} \tbtau(\tilde \eta_n) \delta(|\tilde \eta_n|)\,,
\end{equation}
where $\ind(\xi\wee\eta)$ is the order of
\(\xi\wee\eta\) ({\it i.e.} \(\ind(\xi\wee\eta)=m \iff \xi_m=\eta_m=\xi\wee\eta\)).

If $\dtau$ or $\tdtau$ is continuous then the corresponding formula extends to any \(\xi, \eta\in \Xi\).
\end{lemma}
\begin{proof}
As in \cite{KS10}, Lemma~4.1, with the obvious adaptation in the case of privileged horizontal edges.
\end{proof}

Notice that a sufficient condition for $\dtau$ or $\tdtau$ to be continuous is that
\(\sup_{\xi} \sum_{n} \delta(|\xi_n|) < +\infty\) or \(\sup_{\xi} \sum_{n} \delta(|\tilde \xi_n|) < +\infty\), respectively, (see \cite{KS10} Corollary 4.2).

\begin{prop}
\label{prop-old-new}
Suppose that the subshift is right-special balanced.
\begin{enumerate}[(i)]

\item For all \(\xi,\eta\in \mbox{\rm Im}(\tau)\), we have \(\dtau(\xi,\eta)\leq \tdtau(\xi,\eta)\).

\item Suppose that 
the function
\(\tilde H^{(0)}\ni u\mapsto \frac{\delta(|u|)}{\delta(|\phi^{(0)}(u)|)}\in \RM^+\)
is bounded.
Then the restrictions of $\dtau$ and $\tdtau$ to the graph $\Gamma_\tau=\tilde\Gamma_\tau$ are Lipschitz equivalent.
In particular, if $\dtau$ and $\tdtau$ are continuous then they are Lipschitz equivalent.

\end{enumerate}
\end{prop}
\begin{proof}
We have \(\tbtau(\tilde \xi_n) = 1 \Leftrightarrow \btau( \xi_n)=1\), because $\phi^{(0)}$ is an isomorphism and \(\phi^{(0)}(\tilde \xi_n) = \xi_n\).
Furthermore \(\tilde \xi_n \preceq \phi^{(0)}(\tilde\xi_n) =\xi_n\) so \(\delta(|\xi_n|) \le \delta(\tilde\xi_n)\).
Hence equations~\eqref{eq-dtau} and~\eqref{eq-tdtau} imply that the restrictions to the graph satisfy \(\dtau\leq \tdtau\).

Since the subshift is right-special balanced we also must have
\[
\tilde \xi_n \preceq \xi_n \prec \tilde \xi_{n+1}\,,
\]
for all $n$ and all $\xi$.
Furthermore,
\(
\btau(\xi_n) = \tbtau(\tilde \xi_n),
\)
which directly implies that
\[
\tdtau(\xi,\eta) \leq C \dtau(\xi,\eta)\,
\]
where \(C = \sup_u \frac{\delta(|u|)}{\delta(|\phi^{(0)}(u)|)}\).
\end{proof}

The above Proposition~\ref{prop-old-new} allows us to compare our present work with our previous results in \cite{KS10}.
For right-special balanced subshifts with a weight function satisfying the condition given in (ii), both approaches are equivalent.
Indeed we will prove in Section~\ref{pf11.sect-powerfree}, Theorem~\ref{pf11.thm-characterization}, that a subshift has bounded powers if and only if the infimum and supremum of $\tdtau$ over $\tau$ are Lipschitz equivalent.

An interesting question is to determine which right-special balanced subshifts fulfil condition (ii) in Proposition~\ref{prop-old-new}.
We answer this for Sturmian subshifts.
Sturmian subshifts have a unique right-special word per length, hence are right-special balanced.
It is well-known that for these subshifts bounded powers is equivalent to linear recurrence, see for instance \cite{Dur00,Len03,KS10}. Here, linear recurrence means that there exist a constant $C$ such that the gap between two consecutive occurrences of a word is bounded by $C$ times its length.
\begin{lemma}
A Sturmian subshift satisfies condition (ii) in Proposition~\ref{prop-old-new} if and only if it is linearly recurrent.
\end{lemma}
\begin{proof}
We use the notations of  e.g. \cite{DL03}: $a_n, n\ge 0$, is the $n$-th coefficient in the continuous fraction expansion of the irrational associated to the Sturmian. As is well known linear recurrence (or bounded powers) is equivalent to \(\sup_n a_n < +\infty\) (see e.g. \cite{Len03} Theorem~1 or \cite{KS10} Lemma~4.9).
We write the subshift over the alphabet \(\{0,1\}\), and set \(s_0=0, s_1=0^{a_1-1} 1\), \(s_n=s_{n-1}^{a_n}s_{n-2}, n\ge 2\), and $q_n=|s_n|$.

Consider $u_n=s_{n-1}s_n$.
Words of this type have the longest possible first returns, and since $\delta$ is decreasing it is enough to consider these words to compute the supremum in condition (ii) of Proposition~\ref{prop-old-new}.
The complete first returns to $u_n$ are $v_n = u_n s_n^{a_{n+1}-1}u_n$ and $v'_n= u_n s_n^{a_{n+1}}u_n$.
The word $v_n \wedge v'_n = u_n s_n^{a_{n+1}-1}$ is right-special, and since the subshift is right-special balanced, one has
\[
\phi^{(0)}(u_n) =  u_n s_n^{a_{n+1}-1}\,.
\]
One therefore has:
\[
\frac{|\phi^{(0)}(u_n)|}{|u_n|} =  1 + (a_{n+1}-1)\frac{q_n}{q_n+q_{n-1}}\,,
\]
and gets the inequalities
\[
 \delta(a_{n+1}|u_n|) \le \delta(|\phi^{(0)}(u_n)|) \le \delta(\frac{a_{n+1}+1}{2}|u_n|)\,.
\]
Let $m_n$ be the integer such that \(2^{m_n-1} < a_{n+1} \le 2^{m_n}\).
Using properties (ii) and (i) of the weight $\delta$ in Definition~\ref{def-weight}, one respectively gets
\[
 \underline{c}^{m_n} \delta(|u_n|) \le \delta(a_{n+1}|u_n|) \quad \text{\rm and } \quad
\delta(\frac{a_{n+1}+1}{2}|u_n|) \le \overline{c}\delta(\frac{a_{n+1}+1}{2})\delta(|u_n|)\,,
\]
(notice that $0<\underline{c}<1$) and substituting in the previous inequalities yields
\[
 \frac{1}{\overline{c}\delta(\frac{a_{n+1}+1}{2})} \le \frac{\delta(|u_n|)}{\delta(|\phi^{(0)}(u_n)|)} \le \frac{1}{\underline{c}^{m_n}} \,.
\]
Now if the subshift is lineraly recurrent, then \(\sup_n a_n <+\infty\) and thus \(\sup_n m_n <+\infty\) and condition (ii) of Proposition~\ref{prop-old-new} follows from the above right inequality.
If condition (ii) of Proposition~\ref{prop-old-new} holds, then the above left inequality imply \(\sup_n 1/\delta(a_n) <+\infty \) and it follows that \(\inf_n\delta(a_n) > 0\) and so \(\sup_n a_n <+\infty\) which proves linear recurrence.
\end{proof}

\subsection{Criterion for Lipschitz equivalence}
\label{pf11.ssect-Lipequiv}

We consider now the infimum and supremum of the metrics over all choice functions:
\begin{equation}
\dinf := \inf_\tau d_\tau\,, \qquad \qquad \dsup := \sup_\tau d_\tau\,.
\end{equation}
and
\begin{equation}
\tdinf = \inf_\tau \tilde d_\tau\,, \qquad \qquad \tdsup = \sup_\tau \tilde d_\tau\,.
\end{equation}
Lemma~\ref{pf11.lem-specdist} allows us to obtain explicit formulas.
\begin{prop}
\label{prop-inf}
We have
\[
\dinf(\xi, \eta) = \delta(|\xi\wedge\eta|)\,, \qquad \text{\rm and} \qquad
\tdinf(\xi, \eta) = \delta(|\xi\wee\eta|)\,.
\]
In particular, both metrics induce the topology.
\end{prop}
\begin{proof}
The formulas are proven as in \cite{KS10}, Corollary 4.5, and the latter statement follows from
Lemma~\ref{pf11.lem-infinitepaths}.
\end{proof}
\begin{prop}
\label{pf11.prop-dinfsup}
For any \(\xi, \eta\in \Xi\) we have
\begin{equation}
\dsup(\xi, \eta) = \delta(|\xi\wedge\eta|) + \sum_{n>|\xi\wedge\eta|} \delta(|\xi_n|)
+ \sum_{n>|\xi\wedge\eta|} \delta(|\xi_n|)
\end{equation}
and
\begin{equation}
\tdsup(\xi, \eta) = \delta(|\xi\wee\eta|) + \sum_{n>\ind(\xi\wee\eta)} \delta(|\tilde \xi_n|)
+ \sum_{n>\ind(\xi\wee\eta)} \delta(|\tilde \xi_n|) \,.
\end{equation}
In particular, $\tdinf$ and $\tdsup$ are Lipschitz equivalent if and only if there exists $C>0$ such that for all $\xi \in \Xi$ and all $m$ we have
\begin{equation}
\label{pf11.eq-equivlip}
\delta(|\tilde \xi_m|)^{-1} \sum_{n>m} \delta(|\tilde \xi_n|) \le C
\end{equation}
\end{prop}
\begin{proof} As in \cite{KS10}, Corollary~4.4, with the added remark that by continuity of $\tdinf$ (Proposition~\ref{prop-inf}) the inequality (\ref{pf11.eq-equivlip}) implies the continuity of $\tdsup$.
\end{proof}

\section{Characterization of bounded powers}
\label{pf11.sect-powerfree}

As mentioned in the introduction, the characterization of power boundedness hinges on a comparison of
$\tdinf$ with $\tdsup$. We follow again here closely \cite{KS10} replacing right-special horizontal edges by
privileged horizontal edges.
We state our main theorem.

\begin{thm}
\label{pf11.thm-characterization}
Let $\Xi$ be a minimal and aperiodic subshift over a finite alphabet.
Then $\Xi$ has bounded powers if and only if $\tdsup$ and $\tdinf$ are Lipschitz equivalent.
\end{thm}
\begin{proof}
By Lemma~\ref{pf11.lem-repulsive} we can assume that $\Xi$ is a right-infinite subshift: if $\Xi$ is bi-infinite we consider its right-infinite restriction $\Xi^+$, if $\Xi$ is left-infinite we simply consider its right-infinite ``mirror image''.

Up to rescaling the weight function $\delta$, we can assume that \(\overline{c}=1\), and that $\delta(1) \le 1$.

Assume that $\Xi$ has bounded powers, with index of repulsiveness $\ell>0$.
Fix $\xi\in \Pi_\infty$ and $m \in \NM$.
By definition of privileged words, $\tilde\xi_n$ is a prefix and suffix of $\tilde\xi_{n+1}$, so we have \((|\tilde\xi_{n+1}| - |\tilde\xi_n|) / |\tilde\xi_n| \ge \ell\), and therefore \( |\tilde\xi_{m+k}| \ge (\ell+1)^k |\tilde\xi_{m}|\) for all $k\ge 1$.
The series in equation~\eqref{pf11.eq-equivlip} in Proposition~\ref{pf11.prop-dinfsup} can then be bounded as follows
\[
\delta(|\tilde\xi_{m}|)^{-1} \sum_{n>m} \delta(|\tilde\xi_n|) \le
\frac{1}{\delta(|\tilde\xi_{m}|)} \sum_{k>1} \delta( (\ell+1)^k |\tilde\xi_m| )
\le \sum_{k>1} \delta(\ell + 1)^k \,
\]
where the last inequalities follow from condition (i) in Definition~\ref{def-weight} of a weight function.
The right-hand-side is a convergent geometric series ($\delta(\ell+1) < 1$) and gives a uniform constant to apply Proposition~\ref{pf11.prop-dinfsup} and conclude that $\tdsup$ and $\tdinf$ are Lipschitz equivalent.

Assume now that $\Xi$ does {\em not} have bounded powers.
Fix an odd integer $p=2q+1$ (large).
By Remark~\ref{pf11.rem-powers} there exists a word $u\in\Ll$ such that $u^p\in\Ll$.
By Lemma~\ref{pf11.lem-powers}, there are $p$ (non-empty) privileged words $v_1, \cdots v_p$, such that \(v_1 \prec v_2 \prec \cdots v_{p-1} \preceq u^p \prec v_p\).
Pick an infinite word $\xi$ with prefix $v_p$, and write $m=|v_q|$.
We have
\[
\delta(|\tilde\xi_m|)^{-1} \sum_{n>m} \delta(|\tilde\xi_n|) \ge \delta(|v_q|)^{-1} \sum_{j=q+1}^{2q} \delta(|v_j|) \ge
\delta(|v_q|)^{-1} \; q \; \delta(|v_{2q}|) \ge \underline{c} \,q \,,
\]
where the last inequalities follow from (ii) in Definition~\ref{def-weight} of a weight function.
Since $p$, hence $q$, was chosen arbitrarily large, the criterion for Lipschitz equivalence of Proposition~\ref{pf11.prop-dinfsup} cannot be satisfied, and we conclude that $\tdsup$ and $\tdinf$ are not Lipschitz equivalent.
\end{proof}

\section{Zeta-functions and complexity}
\label{pf11.sect-Zeta}

We define the following zeta-functions, $k\in\NM$:
\[
\zeta_k(s) :=  \sum_{v\in\Tt^{(0)}} a(v)^k \delta(|v|)^s\,,
\qquad \text{\rm and} \qquad
\tilde \zeta_k(s) :=  \sum_{v\in\Tt^{(0)}} \tilde a(v)^k
\delta(|v|)^s\,,
\]
where we use the convention $0^0=0$. One expects that the sums converge for $\Re(s)$ sufficiently large and calls the smallest $s_0$ such that the series converges for $\Re(s)>s_0$ the abscissa of convergence for the series.
The functions have the following interpretations:
\begin{itemize}
\item $\frac12(\zeta_2(s)+\zeta_1(s))$ and \(\frac12(\tilde\zeta_2(s)+\tilde\zeta_1(s))\) are the zeta-functions of
the spectral triples defined by right-special and by privileged words, respectively, see Section~\ref{pf11.sect-spectrip} equation~\eqref{pf11.eq-zeta} and \eqref{pf11.eq-tzeta}.

\item $\zeta_1$ which was denoted $\frac12\zeta_{low}$ in \cite{KS10} (see Section~5.1) is related to the word complexity of the subshift.
Indeed, if we denote by $p(n)$ the number of words of length $n$ then
\[
 \zeta_1(s) = \sum_n (p(n+1)-p(n))\delta(n)^s\,,
\]
and if the complexity has a weak complexity exponent $\beta$ (which is the case, if the upper and the lower box counting dimension of the subshift space exist and the complexity is polynomially bounded, see~\cite{KS10} Section~1.2, and Lemma~5.4 in Section~5.1) then the abscissa of convergence of  $ \zeta_1(s)$ equals $\beta$ (we assume that \(\delta\in\ell^{1+\epsilon}\backslash\ell^{1-\epsilon}\) for all $\epsilon>0$, see \cite{KS10} Section~5.1).

\item $\zeta_0$ and $\tilde \zeta_0$ are related to the complexity $\prs$ of right-special words and the complexity $\ppr$ of privileged words, respectively:
\[
\zeta_0(s) = \sum_n \prs(n)  \delta^s(|v|),\quad
\tilde\zeta_0(s) = \sum_n \ppr(n)  \delta^s(|v|)\,.
\]
If these complexities have weak complexity exponents $\brs$ or $\bpr$ then the abscissa of convergence for $\zeta_0$ or $\tilde \zeta_0$ are $\brs+1$ or $\bpr+1$, respectively.
\end{itemize}
Given that $a(v)$ is bounded we have $\zeta_0(s)\leq \zeta_k(s) \leq |\Aa|^k \zeta_0(s)$ and hence all $\zeta_k$ have the same abscissa of convergence.

Thanks to Lemma~\ref{lem-av} we can compare $\zeta_k$ to $\tilde\zeta_k$.
\begin{prop}
\label{pf11.prop-zeta}
We have $\tilde\zeta_k\geq \zeta_k$ and
\(\zeta_1(s)\geq\frac12 \tilde\zeta_0(s) -\frac12 \delta(0)^s\).

In particular, if the subshift has almost finite rank and \(\delta\in\ell^{1+\epsilon}\backslash\ell^{1-\epsilon}\) for all $\epsilon>0$, then all zeta-functions have the same abscissa of convergence.
\end{prop}
\begin{proof}
We start with the first inequality.
For a privileged word $u$, we let $R(u)$ denote the set of its complete first returns, and let $S(u)$ denote the set of all right-special words $r$, for which there exists $u'\in R(u)$ such that \( u \preceq r \prec u'\).

By Lemma~\ref{lem-av} we have $\tilde a(u)^k \geq \sum_{r\in S(u)} a(r)^k$.
Furthermore $\delta(|u|)\geq \delta(|r|)$ for any $r\in S(u)$.
Hence
\[
\sum_u a(u)^k\delta(|u|)^s \geq \sum_{u\in \tilde H^{(0)}} \sum_{r\in S(u)}
a(r)^k\delta(|r|)^s = \sum_r a(r)^k\delta(|r|)^s.
\]
As for the second inequality we first order the elements of $S(u)$ in such a way that a right-special word which is a prefix of another one comes later in the order.
Let's say we find $r_1$ up to $r_m$.
We now choose first the $a(r_1)$ shortest elements $u'_1,\cdots,u'_{a(r_1)}\in R(u)$ with $r\prec u'_k$, and we replace $a(r_1)\delta(|r_1|)^s$ in the sum for $\zeta_1$ by the smaller term $\sum_{i=1}^{a(r_1)}\delta(|u'_i|)^s$.
We take out these chosen elements of $R(u)$ to obtain $R_1(u)$ and repeat the procedure with $r_2$, that is, choose the $a(r_2)$ shortest elements $u'_1,\cdots,u'_{a(r_2)}\in R_1(u)$ which satisfy $r_2 \prec u'_k$, and take those chosen elements of $R_1(u)$ to obtain $R_2(u)$.
Iterating this construction yields the inequality
\[
\sum_{r\in S(u)}
a(r)^k\delta(|r|)^s \geq \sum_{u'\in R(u)\backslash R_m(u)} \delta(|u'|)^s\,.
\]
$R_m(u)$ has exactly one element left (one of the longest returns to $u$), which we call $v'$.
Then
\[
\delta(|v'|)^s \leq\frac12 \sum_{u'\in R(u)\backslash R_m(u)} \delta(|u'|)^s
\]
and hence
\[
  \sum_{r\in S(u)}
a(r)^k\delta(|r|)^s \geq \frac12  \sum_{u'\in R(u)} \delta(|u'|)^s\,.
\]
Summing up one obtains \(\zeta_1(s) \geq \frac12 \tilde\zeta_0(s) - \frac12\delta(0)^s\).

Now if the subshift has almost finite rank (see Section~\ref{pf11.ssect-spwords}), then for any $v\in \tilde H^{(0)}$,  $\tilde a(v)$ is bounded by $a \log(|v|)^b$
for some uniform constants $a,b>0$.
Since the summability of $\sum_{v\in \tilde H^{(0)}} \delta(|v|)^s $ implies the summability of
$\sum_{v\in \tilde H^{(0)}} \log(|v|)^b \delta(|v|)^{s+\epsilon} $ for any $\epsilon>0$ we see that
$\tilde \zeta^k$ has an abscissa of convergence which does not depend on $k$.
It then follows from the first formulas of the lemma that all zeta-functions have the same abscissa of convergence.
\end{proof}
The last lemma yields immediately relations between the various weak exponents.
\begin{cor}
\label{pf11.cor-zeta}
Assume the existence of weak complexity exponents. Then
\[
\bpr \leq \brs = \beta-1
\]
and there is equality if the subshift has almost finite rank: $\bpr = \brs$.
\end{cor}
The latter result can be seen as an asymptotic version of a much more precise equation between $\ppr$ and $\prs$ which has been obtained for rich subshifts in~\cite{GJWZ09}, namely
\[
\ppr(n)+\ppr(n+1) = \prs(n) + 2 \,,
\]
as in this case privileged words exactly coincide with palindromes by Proposition~\ref{pf11.prop-icr}.

\section{Spectral triples}
\label{pf11.sect-spectrip}
In this final section we provide the spectral triples which can be defined from the graphs $\Gamma_\tau$ and $\tilde \Gamma_\tau$ yielding via Connes' formula the metrics $d_\tau$ and $\tilde d_\tau$ and having zeta-functions related to the ones we introduced above. The first spectral triple corresponds to the construction given in \cite{KS10}.

Consider the C$^\ast$-algebra $C(\Xi)$ of continuous functions on $\Xi$.
Both spectral triples are over $C(\Xi)$ which means that they are given by
\begin{itemize}
\item a representation $\pi_\tau$ resp.\ $\tilde\pi_\tau$ of that algebra on a Hilbert space
$\Hh$ and $\tilde\Hh$ resp.,
\item a self adjoint (unbounded) operator $D$ resp.\ $\tilde D$ of compact resolvent such that
the commutator $[D, \pi_\tau(f)]$ and $[\tilde D, \tilde \pi_\tau(f)]$, resp.\ are bounded for a dense sub-algebra of $C(\Xi)$.
\end{itemize}
Here   the Hilbert spaces are given by $\Hh = \ell^2(E)$ and $\tilde \Hh = \ell^2(\tilde E)$ (where $E$, $\tilde E$ are the edges of the approximation graphs $\Gamma_\tau$ and $\tilde \Gamma_\tau$ defined in Section~\ref{pf11.ssect-approxgraph}) and
the corresponding representations $\pi_\tau$, $\tilde \pi_\tau$ , and Dirac operators $D$, $\tilde D$, by
\begin{equation}
\label{pf11.eq-reprDirac}
\left\{
\begin{array}{lll}
\pi_\tau(f)\phi(e) & = & f( s(e) )\, \phi(e) \\
 D\phi(e) &=& l(e)^{-1}\phi(e^\op)
\end{array}
\right. \qquad \text{\rm and} \qquad
\left\{
\begin{array}{lll}
\tilde\pi_\tau(f) \psi(e) & = & f( s(e) )\, \psi(e) \\
\tilde D \psi(e) &=& l(e)^{-1}\psi(e^\op)
\end{array}
\right.
\end{equation}
for $f\in C(\Xi)$, $\phi \in \Hh$ $\psi\in \tilde\Hh$, $e \in E$ or $\tilde E$, and we recall that for an edge $e=(\xi,\eta)$ we write $e^\op = (\eta,\xi)$.
Notice that the commutators of the Dirac operators with the representations read
\begin{equation}
\label{pf11.eq-commutator}
[D, \pi_\tau(f)] \phi (e) = \frac{f(s(e)) - f(r(e)) }{l(e)} \phi(e^\op) \,, 
\end{equation}
and
\begin{equation}
\label{pf11.eq-tcommutator}
[\tilde D, \tilde \pi_\tau(f)] \psi (e) = \frac{f(s(e)) - f(r(e)) }{l(e)} \psi(e^\op) \,,
\end{equation}
and can be extended to bounded operators on the corresponding Hilbert spaces for all $f$ in the pre-C$^\ast$-algebra of Lipschitz continuous functions over $\Xi$. By definition \cite{Co94} the distances defined by these spectral triples are, resp.\
\begin{equation}
\label{pf11.eq-specdist}
d_\tau (\xi, \eta) = \sup_{f\in C(\Xi)} \bigl\{ |f(\xi) - f(\eta)| \, : \, \| [D,\pi_\tau(f)] \|_{\Bb(\Hh)} \le 1
\bigr\} \,,
\end{equation}
and
\begin{equation}
\label{pf11.eq-tspecdist}
\tilde d_\tau (\xi, \eta) = \sup_{f\in C(\Xi)} \bigl\{ |f(\xi) - f(\eta)| \, : \, \| [\tilde D,\tilde\pi_\tau(f)] \|_{\Bb(\tilde\Hh)} \le 1
\bigr\} \,,
\end{equation}
where \(\| \cdot \|_{\Bb(\Hh)}\) and \(\| \cdot \|_{\Bb(\tilde\Hh)}\) denotes the operator norm on $\Hh$ and $\tilde \Hh$, respectively.
Now formulas \eqref{pf11.eq-commutator} and \eqref{pf11.eq-specdist} directly yield \eqref{def-m1}, while \eqref{pf11.eq-tcommutator} and \eqref{pf11.eq-tspecdist} directly yield (\ref{def-m2}).

\begin{prop}
\label{pf11.prop-spectrip}
Both \(\bigl(C(\Xi), \Hh,  D\bigr)\) and \(\bigl(C(\Xi), \tilde\Hh,  \tilde D\bigr)\) are even spectral triples.
\end{prop}
\begin{proof}
As in \cite{KS10} with simple adaptations for the second spectral triple.
\end{proof}
The zeta-functions of the spectral triples are given by the traces
\begin{equation}
\label{pf11.eq-zeta}
\zeta_D(s) = \TR_{\Hh} \bigl( |D|^{-s} \bigr) =  \frac12 \sum_{v \in \Tt^{(0)}}  a(v)(a(v)+1) \;  \delta(|v|)^s \,,
\end{equation}
and
\begin{equation}
\label{pf11.eq-tzeta}
\zeta_{\tilde D}(s) = \TR_{\tilde\Hh} \bigl( |\tilde D|^{-s} \bigr) =  \frac12 \sum_{v \in \Tt^{(0)}}  \tilde a(v)(\tilde a(v)+1) \;  \delta(|v|)^s \,,
\end{equation}
where $a(v)$ and $\tilde a(v)$ are given in Definition~\ref{pf11.def-av}. Again one expects convergence for sufficiently large real part of $s$. A direct comparison yields that, indeed,
$\zeta_D(s) = \frac12(\zeta_2(s)+\zeta_1(s))$ and \(\zeta_{\tilde D}(s)= \frac12(\tilde\zeta_2(s)+\tilde\zeta_1(s))\).





\end{document}